\documentclass[11pt, leqno]{amsart}
\usepackage{graphicx}
\usepackage{amsfonts,delarray,amssymb,amsmath,amsthm,a4,a4wide}
\usepackage{latexsym}
\usepackage{epsfig}
\usepackage{color}
\usepackage[margin=0.98in]{geometry} 

\newtheorem{thm}{Theorem}[section]

\newtheorem{lem}[thm]{Lemma}
\newtheorem{prop}[thm]{Proposition}
\newtheorem{rem}[thm]{Remark}

\theoremstyle{definition}

\numberwithin{equation}{section}

\newcommand{\R}{\mathbb R}

\newcommand{\e}{\varepsilon}

\newcommand{\p}{\partial}
\newcommand{\cof}{\mbox{cof}\,}
\newcommand{\dist}{\mbox{dist}\,}
\newcommand{\diam}{\mbox{diam}\,}

\newcommand{\comment}[1]{}
\def\h{\hspace*{.24in}}

\newenvironment{myindentpar}[1]%
{\begin{list}{}%
         {\setlength{\leftmargin}{#1}}%
         \item[]%
}
{\end{list}}

\begin{document} 
 \dedicatory{Dedicated to Professor John Urbas on the occasion of his 60th birthday}
\title[ Uniqueness for a system of Monge-Amp\`ere equations]{Uniqueness for a system of Monge-Amp\`ere equations 
}
\author{Nam Q. Le}
\address{Department of Mathematics, Indiana University, 831 E 3rd St,
Bloomington, IN 47405, USA}
\email{nqle@indiana.edu}
\thanks{The research of the author was supported in part by the National Science Foundation under grant DMS-1764248.}


\begin{abstract}

In this note, we prove a uniqueness result, up to a positive multiplicative constant, for nontrivial convex solutions to a system of Monge-Amp\`ere equations 
 \begin{equation*}
\left\{
 \begin{alignedat}{2}
   \det D^2 u~& = \gamma |v|^p~&&\text{in} ~  \Omega, \\\
    \det D^2 v~& = \mu |u|^{n^2/p}~&&\text{in} ~  \Omega, \\\
u=v &= 0~&&\text{on}~ \p\Omega
 \end{alignedat} 
  \right.
  \end{equation*}
  on  bounded, smooth and uniformly convex domains $\Omega\subset \R^n$ provided that $p$ is close to $n\geq 2$. When $p=n$, we show that the uniqueness holds for
  general bounded convex domains $\Omega\subset \R^n$.

\end{abstract}
\maketitle

\section{Introduction and  statement of the main results}

In this note, we are interested in uniqueness issues for
the following system of Monge-Amp\`ere equations on a bounded open convex domain $\Omega \subset\R^n$ $(n\geq 2)$ with positive constants $p, \gamma, \mu$ and convex functions $u$ and $v$:

\begin{equation}
\label{MASp}
\left\{
 \begin{alignedat}{2}
   \det D^2 u~& = \gamma |v|^p~&&\text{in} ~  \Omega, \\\
    \det D^2 v~& = \mu |u|^{n^2/p}~&&\text{in} ~  \Omega, \\\
u=v &= 0~&&\text{on}~ \p\Omega.
 \end{alignedat} 
  \right.
  \end{equation}
When $\Omega$ is a bounded, smooth and uniformly convex domain, Zhang-Qi \cite[Theorem 1.5]{ZQ} show that (\ref{MASp}) has nontrivial convex solutions $u$ and $v$  if and only if $\gamma$ and $\mu$ satisfy
\begin{equation}
\label{cd1}
\gamma \mu^{p/n}= C(n, p,\Omega)
\end{equation}
for some positive constant $C(n,p,\Omega)$. Throughout, by solutions of the Monge-Amp\`ere equations, we always mean their convex solutions in the  sense of Aleksandrov; see \cite{Fi, G} for more details.

One can view (\ref{cd1}) as a sort of uniqueness result for the constants $\gamma$ and $\mu$. A particular corollary of  this analysis (see \cite[Corollary 1.6]{ZQ}) when $p=n$ is that the system of Monge-Amp\`ere equations
\begin{equation}\label{MAS}
\left\{
 \begin{alignedat}{2}
   \det D^2 u~& = \mu |v|^n~&&\text{in} ~  \Omega, \\\
    \det D^2 v~& = \mu |u|^n~&&\text{in} ~  \Omega, \\\
u=v &= 0~&&\text{on}~ \p\Omega
 \end{alignedat} 
  \right.
  \end{equation}
with $\mu>0$ has nontrivial convex solutions $u$ and $v$ on a bounded, smooth and uniformly convex domain $\Omega$ if and only if $\mu$ is the Monge-Amp\`ere eigenvalue of the domain $\Omega$. 

One crucial point in Zhang-Qi's proof of their Theorem 1.5 in \cite{ZQ} is the global Lipschitz regularity for solutions to the Monge-Amp\`ere equations on smooth and uniformly convex domains with globally continuous right hand side and zero boundary data.  With this global regularity, Zhang and Qi were able to apply the boundary Hopf lemma in their fixed point argument using decoupling technique to verify the conditions of a generalized Krein-Rutman theorem developed in Jacobsen \cite{J}, thereby obtaining the existence of solutions to (\ref{MASp}).

An interesting question that was left open in the analysis of \cite{ZQ} is the uniqueness of nontrivial convex solutions $u$ and $v$ to (\ref{MASp}) when $\gamma$ and $\mu$ satisfy (\ref{cd1}). Here, uniqueness should be interpreted as up to a positive multiplicative  constant, for if $u$ and $v$ solve (\ref{MASp}) then
$\tau^{p/n} u$ and $\tau v$ also solve (\ref{MASp}) for any positive constant $\tau>0$. This question is motivated by the following uniqueness results for Monge-Amp\`ere equations:
\begin{myindentpar}{1cm}
(1) The single equation analogue of (\ref{MAS}), that is the Monge-Amp\`ere  eigenvalue problem, has uniqueness of solutions. This was shown by Lions \cite{Ls} for smooth and uniformly convex domains and by the author \cite{L} for general bounded convex domains.\\
(2)  The single equation analogue of (\ref{MASp}), that is the degenerate Monge-Amp\`ere  equation for $0<p\neq n$
\begin{equation*}
\left\{
 \begin{alignedat}{2}
   \det D^2 u~& = |u|^p~&&\text{in} ~  \Omega, \\\
u&= 0~&&\text{on}~ \p\Omega,
 \end{alignedat} 
  \right.
  \end{equation*}
  also has uniqueness of nontrivial solutions when $p<n + \e(n)$ for some small $\e(n)>0$. For $0<p<n$, the uniqueness was obtained by Tso \cite{Tso} while for $n<p<n+\e(n)$, the uniqueness was obtained recently by Huang \cite{H}.
\end{myindentpar}

In \cite{L}, the author proved the existence, uniqueness and variational characterization of the Monge-Amp\`ere eigenvalue, and uniqueness of  convex Monge-Amp\`ere eigenfunctions on general bounded convex domains $\Omega\subset\R^n$.
 These results are 
the singular counterpart of those obtained by  Lions \cite{Ls} and Tso \cite{Tso} in the smooth, uniformly convex setting. 
For convenience, we recall part of  \cite[Theorem 1.1]{L} here.

\begin{thm} 
\label{ev_thm}
 Let $\Omega$ be a bounded open convex domain in $\R^n$. Define $\lambda=\lambda[\Omega]$ by
\begin{equation}
 \lambda[\Omega] =\inf\left\{ \frac{\int_{\Omega} |w|\det D^2 w~ dx }{\int_{\Omega}|w|^{n+1}~ dx}: w\in C(\overline{\Omega}),
 ~w~\text{is convex, nonzero in } \Omega,~ w=0~\text{on}~\p\Omega\right\}.
\label{lam_def}
 \end{equation}
 Then, 
 \begin{myindentpar}{1cm}
 (i) There exists a 
 nonzero convex solution $w\in C^{0,\beta}(\overline{\Omega})\cap C^{\infty}(\Omega)$ for all $\beta\in (0, 1)$ to the Monge-Amp\`ere eigenvalue problem
 \begin{equation}
 \left\{
 \begin{alignedat}{2}
   \det D^{2} w~&=\lambda |w|^{n} \h~&&\text{in} ~\Omega, \\\
w &=0\h~&&\text{on}~\p \Omega.
 \end{alignedat}
 \right.
 \label{EVP_eq}
\end{equation}
Thus the infimum in (\ref{lam_def}) is achieved. The constant $\lambda[\Omega]$ is called the Monge-Amp\`ere eigenvalue of $\Omega$ and $w$ is called a Monge-Amp\`ere eigenfunction of $\Omega$.\\
(ii) The eigenvalue-eigenfunction pair $(\lambda, w)$ to (\ref{EVP_eq}) is unique in the following sense: If the pair $(\Lambda, \tilde w)$ 
satisfies $\det D^2 \tilde w =\Lambda |\tilde w|^n$ in $\Omega$ where $\Lambda>0$ is a positive constant and 
$\tilde w\in C(\overline{\Omega})$ is convex, nonzero
with $\tilde w=0$ on $\p\Omega$, then $\Lambda=\lambda$ and $\tilde w=m w$ for some positive constant $m$.\\
\end{myindentpar}
\end{thm}

Our main results regarding the uniqueness of solutions to (\ref{MASp}) are the following.
\begin{thm}
\label{uni_thmp}
Let $\Omega$ be a bounded, open, smooth and uniformly convex domain in $\R^n$.   Then, provided $|p-n|$ is small,  nontrivial convex solutions $u$ and $v$ to (\ref{MASp}) are unique in the following sense:  if $\hat u$ and $\hat v$ are other nontrivial convex solutions to (\ref{MASp}) then there is a positive constant $\tau>0$ such that
 $\hat u= \tau^{p/n} u$ and $\hat v=\tau v$.
\end{thm}
When $p=n$, we show that the uniqueness holds for
  general bounded convex domains $\Omega\subset \R^n$.

\begin{thm}
\label{uni_thm}
 Let $\Omega$ be a bounded open convex domain in $\R^n$.
Assume that $\mu>0$ and nontrivial convex functions $u$ and $v$ satisfy (\ref{MAS}). Then $\mu$ must be the Monge-Amp\`ere eigenvalue of the domain $\Omega$, $u=v$ and $u$ must be a Monge-Amp\`ere eigenfunction of $\Omega$.
\end{thm}

\begin{rem}
From Proposition \ref{uvathm}, we obtain the existence of nontrivial convex solutions to (\ref{MASp}) with a suitable constants $\gamma>0$ and $\mu>0$ when the domain $\Omega$ is only assumed to be bounded and convex. It would be interesting to prove the uniqueness of solutions to (\ref{MASp}) in this nonsmooth setting when $p\neq n$.
\end{rem}
\begin{rem}
\label{gma_rem}
By considering
$$\bar u:= \gamma^{-\frac{1}{n}} C^{\frac{1}{n+ p}} \|v\|^{-p/n}_{L^{\infty}(\Omega)}u,~\bar v:= \|v\|^{-1}_{L^{\infty}(\Omega)}v,~ \sigma:= C^{\frac{n}{n+ p}}(n,p,\Omega),$$
if necessary,
we can assume in the system (\ref{MASp}) that $$\gamma=\mu= \sigma~\text{and }
\|v\|_{L^{\infty}(\Omega)}=1.$$
\end{rem}

We will use this remark throughout this note. Moreover, we will also use the fact that nontrivial convex solutions to (\ref{MASp}) or to (\ref{MAS}) are strictly convex and $C^{\infty}(\Omega)$ on any bounded convex domain $\Omega$; see, for example \cite[Proposition 2.8]{L} for a proof.

We now indicate some ingredients in the proofs of our main results. For Theorem \ref{uni_thm}, we will use the variational characterization of the Monge-Amp\`ere eigenvalue in Theorem \ref{ev_thm} together with a nonlinear integration by parts in \cite{L} which we will recall in Proposition \ref{NIBP}. We will prove Theorem \ref{uni_thmp} by using a contradiction argument and the uniqueness result for the limiting case of $p=n$ in Theorem \ref{uni_thm}. A critical ingredient in this argument will be the global $C^{2,\beta}$ regularity for solutions to (\ref{MASp}). We will establish this result in Theorem \ref{C2thm}.

The rest of the note is organized as follows. In Section \ref{C2sec}, we will establish uniform estimates and global $C^{2,\alpha}$ regularity for solutions to (\ref{MASp}). In Section \ref{nsec}, we will prove Theorem \ref{uni_thm}. The proof of Theorem \ref{uni_thmp} will be given in Section \ref{psec}.

\section{Uniform estimates and global $C^{2,\alpha}$ regularity }
\label{C2sec}
In this section, we establish uniform estimates and global $C^{2,\alpha}$ regularity for solutions to (\ref{MASp}).  For convenience, by using Remark \ref{gma_rem}, we can assume that $$\gamma=\mu=\sigma>0.$$
We start with the following uniform estimates.
\begin{lem}
\label{suplem}
Let $\Omega$ be a bounded open convex domain in $\R^n$ ($n\geq 2$). Let $p>0$.
Assume that $\sigma>0$ and nontrivial convex functions $u$ and $v$ solve the following system of Monge-Amp\`ere equations: 
\begin{equation}
\label{Sa1}
\left\{
 \begin{alignedat}{2}
   \det D^2 u~& = \sigma |v|^p~&&\text{in} ~  \Omega, \\\
    \det D^2 v~& = \sigma |u|^{n^2/p}~&&\text{in} ~  \Omega, \\\
u=v &= 0~&&\text{on}~ \p\Omega.
 \end{alignedat} 
  \right.
  \end{equation}
  Then there exists a positive constant $C(n,p)>0$ such that
\begin{equation}
\label{unp}
C^{-1}(n, p)|\Omega|^{-2} \leq \sigma \leq C(n, p)|\Omega|^{-2}, C^{-1}(n, p) \|v\|_{L^{\infty}(\Omega)} \leq  \|u\|^{\frac{n}{p}}_{L^{\infty}(\Omega)} \leq C(n, p) \|v\|_{L^{\infty}(\Omega)}.
\end{equation}
\end{lem}
\begin{proof}[Proof of Lemma \ref{suplem}]
Under the unimodular affine transformations $T:\R^n\rightarrow\R^n$ with $\det T=1$:
 $$\Omega\rightarrow T(\Omega),~u(x)\rightarrow u(T^{-1}x), v(x)\rightarrow v(T^{-1}x)$$
 the system (\ref{Sa1}), the quantities 
 $\sigma, \|u\|_{L^{\infty}(\Omega)}, \|v\|_{L^{\infty}(\Omega)}~ \text{and}~ |\Omega|$
 are unchanged. Thus, by John's lemma \cite{Jn},
 we can assume that $\Omega$ satisfies
 $$B_R\subset \Omega\subset B_{nR}~\text{for some}~R>0.$$
 Applying inequality (3.1) in \cite{L} to $\hat v:=\frac{v}{ \|v\|_{L^{\infty}(\Omega)}}$, we obtain for some $c(n,p)>0$
 \begin{equation}
 \label{BR20}
 \int_{B_{R/2}} |v|^{p}~dx=\|v\|^p_{L^{\infty}(\Omega)} \int_{B_{R/2}} |\hat v|^{p}~dx \geq c(n, p)\|v\|^p_{L^{\infty}(\Omega)}|\Omega|.
 \end{equation}
  Applying inequality (3.5) in \cite{L} to $\hat u:=\frac{u}{ \|u\|_{L^{\infty}(\Omega)}}$, we obtain for some $c(n)>0$
  \begin{equation}
  \label{BR2}
  \int_{B_{R/2}} \det D^2 u~dx =  \|u\|^n_{L^{\infty}(\Omega)}   \int_{B_{R/2}} \det D^2 \hat u~dx \leq c(n)   \|u\|^n_{L^{\infty}(\Omega)} |\Omega|^{-1}.
  \end{equation}
  Integrating both sides of the first equation of (\ref{Sa1}) over $B_{R/2}$ and then recalling (\ref{BR20})-(\ref{BR2}), we get
  \begin{equation}
  \label{aBR2}
  \sigma c(n, p)\|v\|^p_{L^{\infty}(\Omega)}|\Omega|\leq c(n)   \|u\|^n_{L^{\infty}(\Omega)} |\Omega|^{-1}.
  \end{equation}
On the other hand, applying the estimates at the end of the proof of Lemma 3.1 (i) in \cite{L} to  $\hat u:=\frac{u}{ \|u\|_{L^{\infty}(\Omega)}}$, we obtain
\begin{eqnarray}
\label{detuOm}
  \int_{\Omega} \det D^2 u~dx =  \|u\|^n_{L^{\infty}(\Omega)}   \int_{\Omega} \det D^2 \hat u~dx &\geq & \|u\|^n_{L^{\infty}(\Omega)}   \int_{\{x\in \Omega:\hat u(x) \leq -\frac{1}{2}\}} \det D^2 \hat u~dx\nonumber \\ &\geq& c(n)|\Omega|^{-1} \|u\|^n_{L^{\infty}(\Omega)}. 
\end{eqnarray}
It follows from (\ref{detuOm}) and first equation of (\ref{Sa1}) that
\begin{equation}
\label{aOm}
c(n)|\Omega|^{-1} \|u\|^n_{L^{\infty}(\Omega)}\leq   \int_{\Omega} \det D^2 u~dx = \sigma \int_{\Omega}|v|^p~dx \leq \sigma \|v\|^p_{L^{\infty}(\Omega)}|\Omega|.
\end{equation}
Therefore, (\ref{aBR2}) and (\ref{aOm}) give 
\begin{equation}
\label{aest1}
\sigma c(n, p)\|v\|^p_{L^{\infty}(\Omega)}|\Omega|\leq c(n)   \|u\|^n_{L^{\infty}(\Omega)} |\Omega|^{-1} \leq \sigma \|v\|^p_{L^{\infty}(\Omega)}|\Omega|.
\end{equation}
Similarly, for the second equation of (\ref{Sa1}), we obtain
 \begin{equation}
 \label{aest2}
\sigma c(n, p)\|u\|^{\frac{n^2}{p}}_{L^{\infty}(\Omega)}|\Omega|\leq c(n)   \|v\|^n_{L^{\infty}(\Omega)} |\Omega|^{-1} \leq \sigma \|u\|^{\frac{n^2}{p}}_{L^{\infty}(\Omega)}|\Omega|.
\end{equation}
Now, we can easily deduce from (\ref{aest1}) and (\ref{aest2}) that
\begin{equation*}
C^{-1}(n, p)|\Omega|^{-2} \leq \sigma \leq C(n, p)|\Omega|^{-2}, C^{-1}(n, p) \|v\|_{L^{\infty}(\Omega)} \leq  \|u\|^{\frac{n}{p}}_{L^{\infty}(\Omega)} \leq C(n, p) \|v\|_{L^{\infty}(\Omega)}
\end{equation*}
for some $C(n, p)>0$. The lemma is proved.
\end{proof}
Note that, by \cite[Theorem 1.5]{ZQ}, when $\Omega$ is a bounded, open, smooth and uniformly convex domain in $\R^n$, the system (\ref{Sa1}) has 
 nontrivial convex solutions $u\in C^{1}(\overline{\Omega})$ and $v\in C^{1}(\overline{\Omega})$ with a suitable $\sigma= \sigma(n,p,\Omega)>0$. Using the uniform estimates in Lemma \ref{suplem} and an approximation argument (see, for example, \cite[Proposition 5.2]{L}), we can extend the existence result of (\ref{Sa1}) to 
 general bounded open convex domains in $\R^n$. We record this result in the next proposition.
\begin{prop}
\label{uvathm}
Let $\Omega$ be a bounded open convex domain in $\R^n$ ($n\geq 2$). Let $p>0$. Then there exist 
a constant $\sigma>0$ and nontrivial convex functions $u$ and $v$ solving the 
system of Monge-Amp\`ere equations (\ref{Sa1}).
\end{prop}

Our main result in this section is concerned with global $C^{2,\alpha}$ regularity for the system of Monge-Amp\`ere equations (\ref{Sa1}).
\begin{thm}
\label{C2thm}
Let $\Omega$ be a  bounded, open, smooth and uniformly convex domain in $\R^n$ where $n\geq 2$. Let $p>0$.
Assume that $\sigma>0$ and nontrivial convex functions $u\in C(\overline{\Omega})$ and $v\in C(\overline{\Omega})$ solve the following system of Monge-Amp\`ere equations: 
\begin{equation}
\label{Sa2}
\left\{
 \begin{alignedat}{2}
   \det D^2 u~& = \sigma |v|^p~&&\text{in} ~  \Omega, \\\
    \det D^2 v~& = \sigma |u|^{n^2/p}~&&\text{in} ~  \Omega, \\\
u=v &= 0~&&\text{on}~ \p\Omega.
 \end{alignedat} 
  \right.
  \end{equation}
  Then $u\in C^{2,\beta_1}(\overline{\Omega})$ for all $\beta_1 < \min\{p, \frac{2}{2+ p}\}$ and $v\in C^{2,\beta_2}(\overline{\Omega})$ for all $\beta_2 < \min\{\frac{n^2}{p}, \frac{2}{2+\frac{n^2}{p}}\}.$
\end{thm}
As mentioned in the introduction, the existence of nontrivial convex functions $u\in C^{1}(\overline{\Omega})$ and $v\in C^{1}(\overline{\Omega})$ solving (\ref{Sa2}) with a suitable $\sigma>0$ was obtained in \cite{ZQ}.
\begin{proof}[Proof of Theorem \ref{C2thm}] The proof is similar to that of 
{\it Step 2} in the proof of \cite[Theorem 5.5]{L} which relies on the proof of Theorem 1.3 in Savin \cite{S}. Since our setting of system of Monge-Amp\`ere equations is slightly different, we include some crucial details for completeness.\\
{\it Step 1: Global $C^{2}$ regularity.}

We can assume that $\|v\|_{L^{\infty}(\Omega)}=1.$
Then, Lemma \ref{suplem} gives
$$C^{-1}(n,p)
\leq \|u\|_{L^{\infty}(\Omega)}\ \leq C(n, p)~\text{and }  C^{-1}(n, p)|\Omega|^{-2} \leq \sigma \leq C(n, p)|\Omega|^{-2}$$
for some positive constant $C(n,p)$.

First of all, we obtain, as in \cite[inequalities (7.1) and (7.2)]{L}, from the convexity of $u$ and the boundedness of the right hand side of $\det D^2 u = \sigma|v|^p$ the following estimates
\begin{equation}
\label{ubounds}
   c(n,p,\Omega) \dist(x,\p\Omega) \leq |u(x)|\leq C(n, p,\Omega) \dist(x,\p\Omega)~\text{for all}~x\in\Omega
\end{equation}
for some positive constants $c(n,p,\Omega)$ and $C(n,p,\Omega)$.

It follows from (\ref{ubounds}) that 
if $x_0 \in \p \Omega$ then $0<c(n,p,\Omega) \leq |D u(x_0)| \le C(n,p,\Omega)$. As a consequence, using the smoothness and uniform convexity of $\p\Omega$, we find that
on $\p \Omega$ the function $u$ separates quadratically from its tangent plane at each $x_0\in\p\Omega$, that is,
\begin{equation}
\label{usep}
\rho|x-x_0|^2\leq u(x)-u(x_0)-Du(x_0)\cdot (x-x_0)\leq \rho^{-1}|x-x_0|^2~\text{ for all }x\in\p\Omega
\end{equation}
for some positive constant $\rho=\rho(n, p, \Omega)$. 

Similarly, using the equation $\det D^2 v= \sigma |u|^{\frac{n^2}{p}}$, we also obtain 
\begin{equation}
\label{vbounds}
   c(n,p,\Omega) \dist(x,\p\Omega) \leq |v(x)|\leq C(n, p,\Omega) \dist(x,\p\Omega)~\text{for all}~x\in\Omega
\end{equation}
and that
for each $x_0\in\p\Omega$, the following quadratic separation estimates for $v$ hold:
\begin{equation}
\label{vsep}
\rho|x-x_0|^2\leq v(x)-v(x_0)-Dv(x_0)\cdot (x-x_0)\leq \rho^{-1}|x-x_0|^2~\text{for all }x\in\p\Omega.
\end{equation}
From (\ref{vbounds}) and the boundedness of $\sigma$, we can apply \cite[Proposition 3.5]{S} to the first equation of (\ref{Sa2}) to conclude that $u$ is pointwise $C^{1,1/3}$ at all points on $\p\Omega$, that is,
$$0\leq u(x) - u(x_0)-Du(x_0)\cdot (x-x_0)\leq C(n,p,\Omega) |x-x_0|^{4/3}~\text{for all } x\in\Omega~\text{and all } x_0\in\p\Omega.$$
This implies that
$Du\in C^{1/3}(\p\Omega)$ and that 
\begin{equation}\label{Du13} g(x):= \frac{|u(x)|}{\dist(x,\p\Omega)}
~\text{ has a uniform }C^{1/3}~ \text{modulus of continuity on }\p\Omega. 
\end{equation}
Similarly, $Dv\in C^{1/3}(\p\Omega)$ and that
\begin{equation}\label{Dv13} h(x):= \frac{|v(x)|}{\dist(x,\p\Omega)}
~\text{ has a uniform }C^{1/3}~ \text{modulus of continuity on }\p\Omega. 
\end{equation}
From 
\begin{equation}
\label{uga}
\left\{
 \begin{alignedat}{2}
   \det D^2 u~& = \sigma |v|^p~&&\text{in} ~  \Omega, \\\
u &= 0~&&\text{on}~ \p\Omega.
 \end{alignedat} 
  \right.
  \end{equation}
  together with (\ref{usep}) and (\ref{Dv13}), we can use \cite[Remark 8.2]{S} to conclude that $u\in C^{1,\gamma}(\overline{\Omega})$ for all $\gamma<1$. This implies that 
  \begin{equation}
  \label{gga}
  g\in C^{\gamma}(\overline{\Omega})~\text{for all }\gamma<1
  \end{equation}
  Similarly, we also have
  \begin{equation}
  \label{hga}
  h\in C^{\gamma}(\overline{\Omega})~\text{for all }\gamma<1.
  \end{equation}
Now, by using \cite[Theorem 2.6]{S}, we obtain from (\ref{uga}), (\ref{usep}) and (\ref{hga})  the global
$C^{2}(\overline{\Omega})$ regularity of $u$.  

Similarly, we obtain from (\ref{vsep}) and (\ref{gga}) and 
\begin{equation}
\label{vga}
\left\{
 \begin{alignedat}{2}
   \det D^2 v~& = \sigma |u|^{\frac{n^2}{p}}~&&\text{in} ~  \Omega, \\\
v &= 0~&&\text{on}~ \p\Omega
 \end{alignedat} 
  \right.
  \end{equation}
    the global
$C^{2}(\overline{\Omega})$ regularity of $v$. 

{\it Step 2: Global $C^{2,\beta}$ regularity.}

 A consequence of the global
$C^{2}(\overline{\Omega})$ regularity for $u$ and $v$ in {\it Step 1} is that 
 $g, h\in C^{0, 1}(\overline{\Omega})$. Then the conditions
of Theorem 1.2 in \cite{LS} are satisfied for the equations (\ref{uga}) and (\ref{vga}) and therefore, we can conclude from this theorem 
that $u\in C^{2,\beta_1} (\overline{\Omega})$ for all $\beta_1 < \min \{p, \frac{2}{2+p}\}$ and $v\in C^{2,\beta_2}(\overline{\Omega})$ for all $\beta_2 < \min\{\frac{n^2}{p}, \frac{2}{2+\frac{n^2}{p}}\}.$
\end{proof}
\begin{rem}
\label{Urem}
In the setting of Theorem \ref{C2thm}, if we normalize $\|v\|_{L^{\infty}(\Omega)}=1$, then from \cite[Theorems 1.1 and 1.2]{LS}, we obtain more precise information about  $D^2 u$ near the boundary.
Indeed, the eigenvalues $\lambda_1(D^2u)\leq\cdots \leq \lambda_n(D^2u)$ of the Hessian matrix $D^2 u$ satisfy
$$\lambda_1 \geq c(n,p,\Omega) \dist^p (x, \p\Omega)~\text{and } \lambda_2\geq c(n,p,\Omega)$$
for some positive constant $c(n,p,\Omega)$.
\end{rem}
\section{Proof of Theorem \ref{uni_thm}}
\label{nsec}
In the proof of Theorem \ref{uni_thm}, we will use the following {\it nonlinear integration by parts} established in \cite[Proposition 1.7]{L}.
\begin{prop}
 \label{NIBP} Let $\Omega$ be a bounded open convex domain in $\R^n$.
 Suppose that $u, v\in C(\overline{\Omega})\cap C^5 (\Omega)$ are strictly convex functions in $\Omega$ with $u=v=0$ on $\p\Omega$ and that there is a constant $M>0$ such that
 \begin{equation}\int_{\Omega}(\det D^2 u)^{\frac{1}{n}}  (\det D^2 v)^{\frac{n-1}{n}}~dx\leq M,~\text{and}~\int_{\Omega}\det D^2 v~dx\leq M.
 \label{detuv}
 \end{equation} Then
\begin{equation} \label{term0}\int_{\Omega} |u|\det D^2 v~dx \geq \int_{\Omega} |v|(\det D^2 u)^{\frac{1}{n}} (\det D^2 v)^{\frac{n-1}{n}}~dx.
\end{equation}
\end{prop}
\begin{proof}[Proof of Theorem \ref{uni_thm}]
To simplify notation, let us denote the Monge-Amp\`ere eigenvalue $\lambda[\Omega]$ of $\Omega$ by $\lambda$. Let $w$ be a Monge-Amp\`ere eigenfunction of $\Omega$ as in Theorem \ref{ev_thm}(i). We note that nontrivial convex solutions $u$ and $v$ to (\ref{MAS}) satisfy $|u(x)|>0$ and $|v(x)|>0$ for all $x\in\Omega$.

As in \cite[Proposition 5.3]{L}, we can show that 
for all $\beta\in (0, 1)$, we have $u, v\in C^{0,\beta}(\overline{\Omega})$ with the estimate
 \begin{equation}
 \label{uvLip}
 |u(x)| + |v(x)|\leq C(n, \beta, \text{diam } (\Omega)) [\text{dist}(x,\p\Omega) ]^{\beta}\left(\|u\|_{L^{\infty}(\Omega)}+ \|v\|_{L^{\infty}(\Omega)}\right)~\text{for all}~x\in\Omega.
 \end{equation}
 From the convexity of $u$ and $u=0$ on $\p\Omega$, we have the gradient estimate
\begin{equation}|Du(x)|\leq \frac{|u(x)|}{\dist (x, \p\Omega)} ~\text{for all}~x\in\Omega.
\label{Du_est}
\end{equation}
Using (\ref{uvLip}) and (\ref{Du_est}), we can argue as in the proof of \cite[Lemma 5.7]{L} to obtain
\begin{equation}\int_{\Omega} (\Delta u + \Delta v) |w|^{n-1}~dx\leq C(n,\Omega)\left(\|u\|_{L^{\infty}(\Omega)}+ \|v\|_{L^{\infty}(\Omega)}\right)\|w\|^{n-1}_{L^{\infty}(\Omega)}.
\label{D2est}
\end{equation}
Because $u+ v$ is smooth and convex in $\Omega$, by the Arithmetic-Geometric inequality, we have
$$n(\det D^2 (u+ v))^{\frac{1}{n}}\leq \Delta (u+ v).$$
From (\ref{D2est}), we find that
\begin{eqnarray}\int_{\Omega}(\det D^2 (u+ v))^{\frac{1}{n}} (\det D^2 w)^{\frac{n-1}{n}} ~dx&\leq&\frac{1}{n} \int_{\Omega}\lambda^{\frac{n-1}{n}}\Delta (u+ v)|w|^{n-1}~dx \nonumber \\ &\leq& C(n,\Omega)
\left(\|u\|_{L^{\infty}(\Omega)}+ \|v\|_{L^{\infty}(\Omega)}\right)\|w\|^{n-1}_{L^{\infty}(\Omega)}.
\label{check_IBP}
\end{eqnarray}
{\it Step 1: $\mu\geq \lambda$}. 

By the characterization of $\lambda$ in Theorem \ref{ev_thm}(i) and the first two equations of (\ref{MAS}),
we find that
\begin{eqnarray}
\label{lammu}
\lambda \int_{\Omega}\left(|u|^{n+1} + |v|^{n+1}\right)dx& \leq& \int_{\Omega} |u| \det D^2 udx + \int_{\Omega} |v| \det D^2 vdx\nonumber\\ &=&  \mu \int_{\Omega}\left (|u| |v|^n + |v| |u|^n \right) dx.
\end{eqnarray}
On the other hand, for each $x\in \Omega$, we have
\begin{multline}
\label{uvn}
|u(x)|^{n+1} + |v(x)|^{n+1} - \left(|u(x)| |v(x)|^n + |v(x)| |u(x)|^n \right)\\ = (|u(x)|-|v(x)|)^2 \sum_{i=1}^{n} |u(x)|^{n-i} |v(x)|^{i-1}\geq 0,
\end{multline}
with equality if and only if $|u(x)|= |v(x)|$.

Combining (\ref{lammu}) with (\ref{uvn}), we obtain 
$\mu\geq \lambda$ as claimed.
\vglue 0.2cm
\noindent
{\it Step 2: $\mu\leq \lambda$}.

In this step, we will use the matrix inequality
$$[\det (A+ B)]^{\frac{1}{n}}\geq (\det A)^{\frac{1}{n}} + (\det B)^{\frac{1}{n}} ~\text{for}~ A, B~\text{symmetric, positive definite}$$
with equality if and only if $A= cB$ for some positive constant $c$.

For all $x\in\Omega$, we have from the above inequality and (\ref{MAS}) that
\begin{equation}
\label{lam_uv}
(\det D^2 (u+ v)(x))^{\frac{1}{n}} \geq (\det D^2 u(x))^{\frac{1}{n}}  + (\det D^2  v(x))^{\frac{1}{n}} =  \mu^{\frac{1}{n}} |u(x) + v(x)|
\end{equation}
with equality if and only if $D^2 u(x)= C(x)D^2 v(x)$ for some positive constant $C(x)$.

By (\ref{check_IBP}), we can apply Proposition \ref{NIBP} to $u+ v$ and $w$. 
Applying Proposition \ref{NIBP} to $u+ v$ and $w$ and using (\ref{lam_uv}), we obtain
 \begin{eqnarray*}\int_{\Omega} \lambda |u+ v| |w|^n~dx= \int_{\Omega}|u+v|\det D^2 w ~dx&\geq&\int_{\Omega}(\det D^2 (u+ v))^{\frac{1}{n}} (\det D^2 w)^{\frac{n-1}{n}} |w|~dx\\ &\geq& \int_{\Omega} \mu^{\frac{1}{n}}\lambda ^{\frac{n-1}{n}}|u+ v| |w|^n~dx.
 \end{eqnarray*}
 It follows that $\lambda\geq \mu$.
 
 {\it Step 3: conclusion.}
 
 From {\it Step 1} and {\it Step 2}, we find that $\mu=\lambda$ and
 we must have equalities  in (\ref{uvn}) and (\ref{lam_uv}) for all $x\in\Omega$. It follows that $|u|= |v|$ in $\Omega$. Thus $u=v$ and $u$ solves $\det D^2 u=\lambda |u|^n$ in $\Omega$ with $u=0$ on $\p\Omega$. By Theorem \ref{ev_thm} (ii), $u$ is a Monge-Amp\`ere eigenfunction of $\Omega$.
 \end{proof}

\section{Proof of Theorem \ref{uni_thmp}}
\label{psec}
In this section, we prove the uniqueness result as stated in Theorem \ref{uni_thmp}. Our proof is inspired by that of \cite[Theorem 1.1(2)]{H}.
\begin{proof}[Proof of Theorem \ref{uni_thmp}]
By Remark \ref{gma_rem}, it suffices to prove the uniqueness of nontrivial convex solutions to the system of Monge-Amp\`ere equations:
\begin{equation*}
\left\{
 \begin{alignedat}{2}
   \det D^2 u~& = \sigma |v|^p~&&\text{in} ~  \Omega, \\\
    \det D^2 v~& = \sigma |u|^{n^2/p}~&&\text{in} ~  \Omega, \\\
u=v &= 0~&&\text{on}~ \p\Omega.
 \end{alignedat} 
  \right.
  \end{equation*}
By the symmetry of $p$ and $n^2/p$, it suffices to prove uniqueness for $p-n> 0$ small since the case $p=n$ is covered by Theorem \ref{uni_thm}. 
We argue by contradiction.

Suppose that for a sequence $p_k\searrow n$, the following system of Monge-Amp\`ere equations
\begin{equation}
\label{MASak}
\left\{
 \begin{alignedat}{2}
   \det D^2 u_k~& = \sigma_k |v_k|^{p_k}~&&\text{in} ~  \Omega, \\\
    \det D^2 v_k~& = \sigma_k |u_k|^{n^2/p_k}~&&\text{in} ~  \Omega, \\\
u_k=v_k &= 0~&&\text{on}~ \p\Omega
 \end{alignedat} 
  \right.
  \end{equation}
has at least two distinguished pairs of convex solutions $(u_k, v_k)$ and $(\tilde u_k, \tilde v_k)$ where
\begin{equation}
\label{vk1}
\|v_k\|_{L^{\infty}(\Omega)}=\|\tilde v_k\|_{L^{\infty}(\Omega)}=1.
\end{equation}
We can assume that for all $k$
 \begin{equation}
 \label{pkn}
 n< p_k\leq n+ \frac{1}{2},~\text{and }\|u_k\|_{L^{\infty}(\Omega)}\geq \|\tilde u_k\|_{L^{\infty}(\Omega)}.
 \end{equation}
 Taking a subsequence if necessary, and without loss of generality, we can assume that 
\begin{equation}
\label{taulim}
\lim_{k\rightarrow \infty}  \frac{\|\tilde v_k-v_k\|_{L^{\infty}(\Omega)}}{\|\tilde u_k-u_k\|_{L^{\infty}(\Omega)}}=\tau\in [0,1].
\end{equation}
Let $$\phi_k= \frac{\tilde u_k - u_k}{\|\tilde u_k-u_k\|_{L^{\infty}(\Omega)}},~\text{and } \varphi_k=  \frac{\tilde v_k - v_k}{\|\tilde v_k-v_k\|_{L^{\infty}(\Omega)}}. $$
We will prove (see {\it Step 6}) that  for all $k$ large
$$\phi_k>0, ~\text{and }\varphi_k>0~\text{in }\Omega$$
and this will clearly lead to a contradiction to (\ref{vk1}). Hence, we must have the uniqueness of solutions as stated in the theorem. We now proceed with  proof with several steps.
\vglue 0.2cm
\noindent
{\it Step 1: Convergence of $\sigma_k$ to the Monge-Amp\`ere eigenvalue of $\Omega$ and convergence of $u_k$, $v_k$, $\tilde u_k$ and $\tilde v_k$ in $C^{0,\frac{1}{n}}(\overline{\Omega})$ to the same Monge-Amp\`ere eigenfunction of $\Omega$.}

Recalling (\ref{unp}) together with (\ref{vk1}), and using the Aleksandrov maximum principle
 (see \cite[Theorem 2.8]{Fi} and \cite[Theorem 1.4.2]{G}) and  the compactness of solutions to the Monge-Amp\`ere equation (see \cite[Corollary 2.12]{Fi} and \cite[Lemma 5.3.1]{G}), we find that up to extracting a subsequence, 
 $\sigma_k\rightarrow \sigma$, while $u_k\rightarrow u$ and $v_k\rightarrow v$ uniformly in $C^{0,\frac{1}{n}}(\overline{\Omega})$, and the following system holds

\begin{equation*}
\left\{
 \begin{alignedat}{2}
   \det D^2 u~& = \sigma |v|^n~&&\text{in} ~  \Omega, \\\
    \det D^2 v~& = \sigma |u|^{n}~&&\text{in} ~  \Omega, \\\
u=v &= 0~&&\text{on}~ \p\Omega.
 \end{alignedat} 
  \right.
  \end{equation*}
By Theorem \ref{uni_thm}, we have the uniqueness, that is, $\sigma=\lambda$ is the Monge-Amp\`ere eigenvalue of $\Omega$ and $u=v=w$ is the Monge-Amp\`ere eigenfunction of $\Omega$ with $L^{\infty}$ norm being $1$:

\begin{equation}
\label{wdef}
\left\{
 \begin{alignedat}{2}
   \det D^2 w~& = \lambda |w|^n~&&\text{in} ~  \Omega, \\\
w&= 0~&&\text{on}~ \p\Omega,\\
\|w\|_{L^{\infty}(\Omega)}&= 1.
 \end{alignedat} 
  \right.
  \end{equation}
  By this uniqueness, we actually have the full convergences of $\sigma_k$ to $\lambda$, $u_k$ to $w$ and $v_k$ to $w$ uniformly in $C^{0,\frac{1}{n}}(\overline{\Omega})$ when $k\rightarrow\infty$. Similarly, we also have the full convergences of $\tilde u_k$ to $w$ and $\tilde v_k$ to $w$ uniformly in $C^{0,\frac{1}{n}}(\overline{\Omega})$ when $k\rightarrow\infty$.
  
  We denote by $W=(W^{ij})_{1\leq i, j\leq n}= \cof (D^2 w)$ the cofactor matrix of the Hessian $D^2 w$, so that
  $$W= (\det D^2 w) (D^2 w)^{-1}~\text{in }\Omega.$$
  For later use, we note that for some constant $c(\Omega)>0$
  \begin{equation}
  \label{wcc}
  c(\Omega)\dist(x,\Omega) \leq |w(x)|=-w(x) \leq c^{-1}(\Omega)\dist(x,\Omega).
  \end{equation}
In the next steps, 
   the convex function $\psi\in C^{\infty}(\overline{\Omega})$ solving the Monge-Amp\`ere equation
\begin{equation*}
\left\{
 \begin{alignedat}{2}
   \det D^2 \psi~& = 1~&&\text{in} ~  \Omega, \\\
\psi &= 0~&&\text{on}~ \p\Omega
 \end{alignedat} 
  \right.
  \end{equation*}
  will be very useful in our comparison arguments. 
  
  Observe that
  for some constant $c_0= c_0 (n,\Omega)>0$ 
\begin{equation}
\label{psi_est}
D^2 \psi\geq c_0 I_n,~\text{and } c_0 \dist(x,\p\Omega) \leq |\psi(x)| \leq c^{-1}_0 \dist(x,\p\Omega) ~\text{in }\overline{\Omega}.
\end{equation}
 \vglue 0.2cm
\noindent 
{\it Step 2: Systems of linearized Monge-Amp\`ere equations for $\phi_k$ and $\varphi_k$.} 

Throughout, we will use the following notation:  {\it $f_{ij}=\frac{\p^2 f}{\p x_i \p x_j}$ for a function $f$ and $A_{ij}$ for the $(i, j)$ entry of a matrix A.}

Note that
$$\det D^2 u_k -\det D^2 \tilde u_k = U^{ij}_k (u_k- \tilde u_k)_{ij}~ \text{and}~ (-v_k)^{p_k} -  (-\tilde v_k)^{p_k} =V_k (\tilde v_k - v_k)$$
where
$$U^{ij}_k= \int_{0}^1 [\cof ( t D^2 u_k + (1-t) D^2\tilde  u_k)]_{ij}dt, $$
and
$$V_k= \int_{0}^{1} p_k[- t v_k -(1-t) \tilde v_k]^{p_k-1} dt=\int_{0}^{1} p_k|t v_k +(1-t) \tilde v_k|^{p_k-1} dt .$$
From
$$\det D^2 u_k -\det D^2 \tilde u_k =\sigma_k (-v_k)^{p_k} -\sigma_k (-\tilde v_k)^{p_k} $$
we obtain
$$
-U^{ij}_k (\tilde u_k - u_k)_{ ij} = \sigma_k V_k (\tilde v_k- v_k),
$$
or, $\phi_k$ and $\varphi_k$ satisfies the following linearized Monge-Amp\`ere equation
\begin{equation}
\label{LMAphik}
U^{ij}_k \phi_{k, ij} + \sigma_k V_k \varphi_{k} \frac{\|\tilde v_k-v_k\|_{L^{\infty}(\Omega)}}{\|\tilde u_k-u_k\|_{L^{\infty}(\Omega)}}=0.
\end{equation}
Similarly, we have
\begin{equation}
\label{LMAvarphik}
V^{ij}_k \varphi_{k, ij} + \sigma_k U_k \phi_{k} \frac{\|\tilde u_k-u_k\|_{L^{\infty}(\Omega)}}{\|\tilde v_k-v_k\|_{L^{\infty}(\Omega)}}=0.
\end{equation}
where
$$V^{ij}_k= \int_{0}^1 [\cof ( t D^2 v_k + (1-t) D^2\tilde  v_k)]_{ij}dt ~
\text{ and }
U_k= \int_{0}^{1} \frac{n^2}{p_k}| t u_k +(1-t) \tilde u_k|^{\frac{n^2}{p_k}-1} dt.$$
When $k\rightarrow \infty$, we deduce from {\it Step 1} and Theorem \ref{C2thm} that for $\beta:=\frac{2}{3+ n}$,
\begin{equation}
\label{Vkcon}
V_k\rightarrow n |w|^{n-1},~U_k\rightarrow n|w|^{n-1}~\text{uniformly on } C^{2,\beta}(\overline{\Omega}),
\end{equation}
while
\begin{equation} 
\label{Vijcon}
U^{ij}_k\rightarrow W^{ij},~ V^{ij}_k\rightarrow W^{ij}~\text{uniformly on } C^{\beta}(\overline{\Omega}).
\end{equation}
\vglue 0.2cm
\noindent
{\it Step 3: $|\phi_k(x)|\leq C(n,\Omega)\dist(x,\p\Omega)$ for $k$ large.}

By (\ref{psi_est}), it suffices to show that for all $k$ large
\begin{equation}
\label{phikpsi}
|\phi_k| \leq C(n,\Omega)|\psi|~\text{in }\Omega.
\end{equation}
Indeed, as in (\ref{vbounds}) of the proof of Theorem \ref{C2thm}, we have
$$ c(\Omega)\dist(x,\Omega)\leq |v_k(x)| \leq  C(n,\Omega) \dist(x,\p\Omega)$$
and
$$ c(\Omega)\dist(x,\Omega)\leq |\tilde v_k(x)| \leq  C(n,\Omega) \dist(x,\p\Omega).$$
Therefore, for all $k$, we have
\begin{equation}
\label{Vkdist}
|V_k(x)|\leq p_k C^{p_k-1}(n,\Omega) \dist^{p_k-1}(x,\p\Omega) \leq C_1(n, \Omega) \dist^{n-1}(x,\p\Omega)
\end{equation}
where we used (\ref{pkn}) in the last inequality.

On the other hand, by {\it Step 1} and (\ref{taulim})
$$\sigma_k\leq 2\lambda,\quad \frac{\|\tilde v_k-v_k\|_{L^{\infty}(\Omega)}}{\|\tilde u_k-u_k\|_{L^{\infty}(\Omega)}} \leq 2\tau+1~\text{for all  large }k .$$
Thus, in view of (\ref{LMAphik}), for all $k$ large,  we have in $\Omega$
\begin{equation}
\label{estk1}
|U^{ij}_k \phi_{k, ij}| = \sigma_k |V_k| |\varphi_k| \frac{\|\tilde v_k-v_k\|_{L^{\infty}(\Omega)}}{\|\tilde u_k-u_k\|_{L^{\infty}(\Omega)}} \leq 2 \lambda(2\tau+1) |V_k|\leq C_2(n,\Omega) \dist^{n-1}(x,\p\Omega) .
\end{equation}
From Remark \ref{Urem}, we infer that the eigenvalues $\lambda_{k, 1}\leq\cdots\leq \lambda_{k,n}$ of $U^{ij}_k$ satisfies for some $c_1=c_1(n,\Omega)>0$
$$\lambda_{k, n}\geq c_1; \lambda_{k, 1}\geq c_1\dist^{p_k} (x,\p\Omega).$$
It follows from the above estimates and (\ref{psi_est}) that 
\begin{equation}
\label{Upsi}
U^{ij}_k \psi_{ij} \geq c_0 \text{trace} (U^{ij}_k) \geq c_0 c_1:= c_2.
\end{equation}
Thus for $C(n,\Omega)$ and $k$ large, we have from (\ref{estk1}) and (\ref{Upsi})
$$U^{ij}_k (-C(n,\Omega)\psi)_{ij}<U^{ij}_k \phi_{k, ij}< U^{ij}_k (C(n,\Omega)\psi)_{ij} ~\text{in }\Omega.$$
Using the maximum principle, we obtain
 (\ref{phikpsi}).
\vglue 0.2cm
\noindent
{\it Step 4: $\tau>0$}.

Indeed, suppose otherwise that $\tau$ defined by (\ref{taulim}) satisfies $\tau=0.$ In this case, we use the result of  {\it Step 3} together with (\ref{Vkcon}) and (\ref{Vijcon}) (in fact, only the locally uniform convergences suffice) to pass to the limit of $k\rightarrow\infty$ in (\ref{LMAphik}).  By {\it Step 3}, we can assume, up to extracting a subsequence, that $\phi_k$ converges locally uniformly in $C^{2,\beta}(\Omega)$ and uniformly in $C^{0,1}(\overline{\Omega})$ to a  Lipschitz function $\phi\in C^{2,\beta}(\Omega)\cap C^{0, 1}(\overline{\Omega})$.  Letting $k\rightarrow\infty$ in (\ref{LMAphik}) and using (\ref{Vkcon}), (\ref{Vijcon}), (\ref{Vkdist}) and $\tau=0$, we find that $\phi$ satisfies
\begin{equation*}
 W^{ij} \phi_{ij} = 0~\text{in} ~  \Omega,~\text{and }
\phi = 0~\text{on}~ \p\Omega.
\end{equation*}
 From $$W^{ij} w_{ij}= n\det D^2 w = n\lambda|w|^{n}>0 \text{ in }\Omega$$
  and the maximum principle, we have $|\phi| \leq \e (-w)$ in $\Omega$ for all $\e>0$.
  This implies  $\phi\equiv 0$. However, this contradicts the fact that
 $\|\phi\|_{L^{\infty}(\Omega)}=1.$
Hence $\tau>0$.
\vglue 0.2cm
\noindent
{\it Step 5: $\phi_k$ and $\varphi_k$ converge uniformly in $C^{0,1}(\overline{\Omega})$ to $|w|$ defined in (\ref{wdef}).}

As in {\it Step 3}, now with $0<\tau\leq 1$, we use $$\lim_{k\rightarrow \infty}  \frac{\|\tilde u_k-u_k\|_{L^{\infty}(\Omega)}}{\|\tilde v_k-v_k\|_{L^{\infty}(\Omega)}}=\frac{1}{\tau}$$ in (\ref{LMAvarphik}) to obtain
$$|\varphi_k| \leq C(n,\Omega)|\psi| \leq C(n,\Omega) \dist(x,\p\Omega).$$
Thus, up to extracting a subsequence, we can assume that $\{\phi_k\}$ and $\{\varphi_k\}$, respectively, converge locally uniformly in $C^{2,\beta}(\Omega)$ and uniformly in $C^{0,1}(\overline{\Omega})$ to Lipschitz functions $\phi\in C^{2,\beta}(\Omega)\cap C^{0, 1}(\overline{\Omega})$ 
and $\varphi\in C^{2,\beta}(\Omega)\cap C^{0, 1}(\overline{\Omega})$, respectively.
Using (\ref{Vkcon}) and (\ref{Vijcon}) together with $\sigma_k\rightarrow \lambda$ in the linearized Monge-Amp\`ere equations (\ref{LMAphik}) and (\ref{LMAvarphik}), we find that these functions $\phi$
and $\varphi$ satisfy
\begin{equation*}
\left\{
 \begin{alignedat}{2}
   W^{ij}\phi_{ij} + \lambda n |w|^{n-1} \tau\varphi~& = 0~&&\text{in} ~  \Omega, \\\
    W^{ij}\varphi_{ij} + \lambda n |w|^{n-1}\frac{\phi}{\tau}~& = 0~&&\text{in} ~  \Omega, \\\
\phi=\varphi &= 0~&&\text{on}~ \p\Omega.
 \end{alignedat} 
  \right.
  \end{equation*}
  Therefore,
  $$ W^{ij}(\phi-\tau\varphi)_{ij} - \lambda n |w|^{n-1}(\phi- \tau\varphi)=0~\text{in }\Omega,~\text{and }\phi-\tau\varphi=0~\text{on }\p\Omega.$$
  As in {\it Step 4}, we use the maximum principle to get $|\phi-\tau\varphi|<\e (-w)$ in $\Omega$ for all $\e>0$.
 It follows that $\phi=\tau\varphi.$ 
 Since $\|\phi\|_{L^{\infty}(\Omega)}=\|\varphi\|_{L^{\infty}(\Omega)}=1,$
 we have $\tau=1$; hence $\phi=\varphi$ and $\varphi$ satisfies
 $$ W^{ij}\phi_{ij} + \lambda n |w|^{n-1} \phi =0\text{ in} ~  \Omega.$$
 Using (\ref{wcc}) and {\it Step 3}, we have $M(-w)-\phi>0$ in $\Omega$ for a large constant  $M>0$. Now, $M(-w)-\phi$ and $-w$ are positive eigenfunctions corresponding to the eigenvalue $\lambda$ of the operator
 $-\frac{W^{ij}}{n|w|^{n-1}} \p_{ij}$ in $\Omega$. Note that $$\det \left(\frac{W^{ij}}{n|w|^{n-1}}\right)= \frac{(\det D^2 w)^{n-1}}{n^n |w|^{n(n-1)}}=\frac{\lambda^{n-1}}{n^n}.$$ It follows that $M(-w)-\phi= \theta (-w)$ for some positive constant $\theta$; see, for example \cite[Proposition A.2]{Ls}.
 Therefore, $\phi = \tau w$ for some constant $\tau$. From $\|\phi\|_{L^{\infty}(\Omega)}=\|w\|_{L^{\infty}(\Omega)}=1,$ we find 
 $$\phi=\varphi =\pm w.$$
 To show that $\phi=|w|$, it suffices to show that the limit function $\phi\geq 0$ at some interior point of $\Omega$. 
 
 Let $x_k\in\Omega$ be a minimum point of $u_k$. Then, from (\ref{vk1}) and Lemma \ref{suplem}, we have $|u_k(x_k)|= \|u_k\|_{L^{\infty}(\Omega)}\geq C^{-1}(n, p).$
 By the Aleksandrov maximum principle
 (see \cite[Theorem 2.8]{Fi} and \cite[Theorem 1.4.2]{G}) and the bound on $\sigma_k$ in Lemma \ref{suplem}, we have
 \begin{equation*}
 |u_k(x_k)|^n\leq C(n) (\diam \Omega)^{n-1} \dist(x_k,\p\Omega) 
 \int_{\Omega} \det D^2 u_k~dx\leq C(n, p, \Omega) \dist(x_k,\p\Omega).
\end{equation*}
This implies that
\begin{equation} 
\label{xksep}
\dist(x_k,\p\Omega) \geq C^{-1}(n,p,\Omega).
\end{equation}
 At $x_k$, by (\ref{pkn}), we have $$\tilde u_k(x_k)- u_k(x_k)=\|u_k\|_{L^{\infty}(\Omega)} + \tilde u_k(x_k)\geq \|u_k\|_{L^{\infty}(\Omega)}-\|\tilde u_k\|_{L^{\infty}(\Omega)}\geq 0$$
 and thus
 $$\phi_k(x_k)\geq 0.$$
 This together with (\ref{xksep}) shows that $\phi(z)\geq 0$ where $z\in\Omega$ is a limit point of $\{x_k\}$. 
 In conclusion, 
 $$\phi=\varphi=-w =|w|.$$
 \vglue 0.2cm
\noindent
 {\it Step 6: $\phi_k>0$ and $\varphi_k>0$ when $k$ is large enough.}

We are going to show if $k$ and $M$ are large, and $\delta>0$ small, then 
$$\eta:=M\delta^{n} \psi -\delta w$$
is a lower barrier for $\phi_k$ in the boundary ring
  $$\Omega_\delta:=\{x\in\Omega|\dist(x,\p\Omega)<\delta\}.$$
  Let
  $c_3:=c(\Omega)/2$
  where $c(\Omega)$ is as in (\ref{wcc}). Then, by {\it Step 5} and (\ref{wcc}), for any fixed $\delta>0$, we can find a large positive integer $k_0=k_0(\delta,\Omega)$ such that 
 \begin{equation}
 \label{phikout}
 \phi_k\geq c_3\delta~\text{in } \Omega\backslash\Omega_\delta~\text{for all } k\geq k_0.
 \end{equation}
 In view of (\ref{Vijcon}), we have the following uniform convergence in $C(\overline{\Omega})$
 $$U^{ij}_k w_{ij}\rightarrow W^{ij} w_{ij}= n\det D^2 w = n\lambda |w|^n \leq C_1(n,\Omega) \dist^n(x,\p\Omega),$$
 which implies that
 $$U^{ij}_k w_{ij}\leq C_1 \dist^n(x,\p\Omega) + \e_k~\text{in }\Omega$$
 where $\e_k\rightarrow 0$ when $k\rightarrow \infty$. 
 
 Therefore, using {\it Step 3} together with (\ref{Upsi}) and (\ref{Vkdist}),
 we have in $\Omega_\delta$
 \begin{eqnarray}
 \label{phiketa}
 U_k^{ij} (\phi_k-\eta)_{ij} &=& U^{ij}_k \phi_{k, ij} -M\delta^{n} U^{ij}_k \psi_{ij} +\delta U^{ij}_k w_{ij}\nonumber\\
 &\leq& 4\lambda \tau |V_k \phi_k| - M\delta^{n} c_2 + \delta C_1 \dist^n(x,\p\Omega) + \delta\e_k\nonumber\\
 &\leq& C_2(n,\Omega)\dist^{n} (x,\p\Omega) - M\delta^{n} c_2 + \delta C_1 \dist^n(x,\p\Omega) + \delta \e_k<0
 \end{eqnarray} 
 provided that $M$ is large (depending only on $n$ and $\Omega$) and $k\geq k_1(\delta,n,\Omega)$ where $k_1$ is large.

On the other hand, for $k\geq  k_1$, using (\ref{phikout}) together with (\ref{wcc}) and (\ref{psi_est}), we have, on $\p\Omega_\delta\backslash\p\Omega$
$$\phi_k-\eta=\phi_k+ M\delta^n |\psi|-\delta |w| \geq c_3 \delta + c_0 M\delta^{n+ 1} -c^{-1} \delta^2>0$$
provided $\delta\leq \delta_0$ where $\delta_0=\delta_0(n,\Omega)>0$ is small.  

Now, it follows from (\ref{phiketa}) and the maximum principle that, for all  $k\geq k_2(\delta, n,\Omega):=\max\{k_0, k_1\}$ and $\delta\leq \delta_0$,
$$\phi_k-\eta\geq 0 ~\text{in }\Omega_{\delta}.$$
Consequently, using (\ref{wcc}) and (\ref{psi_est}) once more time,  we have for all $k\geq k_2$
$$\phi_k \geq \eta=-M\delta^{n}|\psi| +\delta |w| \geq- c_0^{-1} M\delta^{n}\dist(x,\p\Omega) + \delta c \dist(x,\p\Omega) \geq  \frac{c\delta}{2} \dist(x,\p\Omega) ~\text{in }\Omega_{\delta}$$
provided $\delta\leq \delta_1(n,\Omega)$ small. This combined with (\ref{phikout}) shows that $\phi_k>0$ in $\Omega$ for $k$ large enough.

The same argument shows that $\varphi_k>0$ in $\Omega$ for $k$ large enough. This completes the proof of our theorem.
\end{proof} 
\noindent
{\bf Acknowledgements.} The author would like to thank the referees for their helpful comments.

\end{document}